\documentclass[12pt, amscd]{amsart}
\usepackage[letterpaper,hmargin=1.25in,vmargin=1.5in]{geometry}
\usepackage{amscd}
\usepackage{verbatim}
\usepackage{color}

\input xy
\xyoption{all}

\usepackage{amssymb}
\usepackage{hyperref}

\DeclareMathAlphabet{\cat}{OT1}{cmss}{m}{sl}

\newtheorem{theorem}{Theorem}[section]

\newtheorem{proposition}[theorem]{Proposition}
\newtheorem{lemma}[theorem]{Lemma}
\newtheorem{corollary}[theorem]{Corollary}

\theoremstyle{definition}
\newtheorem{remark}[theorem]{Remark}

\newtheorem{example}[theorem]{Example}

\newcommand{\tens}{\otimes}

\newcommand{\CH}{\operatorname{CH}}

\renewcommand{\Im}{\operatorname{Im}}

\newcommand{\Ker}{\operatorname{Ker}}

\newcommand{\Pic}{\operatorname{Pic}}

\newcommand{\ind}{\operatorname{\hspace{0.3mm}ind}}

\newcommand{\res}{\operatorname{res}}

\newcommand{\Br}{\operatorname{Br}}
\newcommand{\Spec}{\operatorname{Spec}}

\newcommand{\SB}{\operatorname{SB}}

\newcommand{\gSL}{\operatorname{\mathbf{SL}}}

\newcommand{\rank}{\operatorname{rank}}

\newcommand{\tors}{\operatorname{\cat{tors}}}

\renewcommand{\P}{\mathbb{P}}
\newcommand{\Z}{\mathbb{Z}}

\title[On the torsion of Chow groups of Severi-Brauer varieties] 
{On the torsion of Chow groups of Severi-Brauer varieties}

\author
[S. Baek] {Sanghoon Baek}

\address
{Department of Mathematical Sciences, KAIST, 291 Daehak-ro, Yuseong-gu, Daejeon, 305-701, Republic of Korea}

\email {sanghoonbaek@kaist.ac.kr}

\begin{document}

\maketitle

\begin{abstract}
In this paper, we generalize a result of Karpenko on the torsion in the second quotient of the gamma filtration for Severi-Brauer varieties to higher degrees. As an application, we provide a nontrivial torsion in higher Chow groups and the topological filtration of the associated generic variety and obtain new upper bounds for the annihilators of the torsion subgroups in the Chow groups of a large class of Severi-Brauer varieties. In particular, using the torsion in higher degrees, we show indecomposability of certain algebras.
\end{abstract}

\section{Introduction}\label{intro}

Let $p$ be a prime and $A$ a central simple algebra of $p$-power degree over a field. We denote by $\SB(A)$ the corresponding Severi-Brauer variety. Consider the Grothendieck ring $K_{0}(\SB(A))$ and its gamma filtration $\Gamma^{d}(\SB(A))$ of degree $d\geq 0$. By a theorem of Quillen, the gamma filtration on $K_{0}(\SB(A))$ is determined by the indices of (tensor) $p$-powers of $A$, whose exponents form a \emph{reduced sequence}. In \cite{Kar}, Karpenko provided a lower bound for the size of the torsion subgroup of the $2$nd quotient $\Gamma^{2/3}(\SB(A))$ of the gamma filtration and showed this bound is sharp if the corresponding reduced sequence has length one. Note that the $0$th and the $1$st quotients are torsion-free.

In the present paper, we provide a general lower bound for the size of the torsion subgroup $\Gamma^{d/d+1}(\SB(A))_{\tors}$, which extends Karpenko's bound. In particular, we determine a list of degrees where the torsion subgroup is nontrivial and show that the lower bound is sharp for certain algebras including the case where $A$ has the doubly reduced sequence of length one.   

We apply the lower bound obtained by the gamma filtration to the torsion subgroups in both topological filtration and the Chow group of the Severi-Brauer variety associated to the corresponding generic division algebra. As a consequence, we find a nontrivial torsion of the Chow groups in higher codimensions. In general, it is wide open to find a nontrivial torsion in higher Chow groups of a projective homogeneous variety: a result for projective quadrics can be seen in \cite{Vishik}. Applying our result on the torsion in higher quotient of the topological filtration, we also show the indecomposability of certain generic algebras, which can be done by using the torsion in the $2$nd quotient of the topological filtration \cite{Kar}.

Consider the torsion subgroup $\CH^{d}(\SB(A))_{\tors}$ of codimension $d$. It is well-known that for $d=0, 1, \dim(\SB(A))$, the torsion subgroups are all trivial. However, in general, the torsion subgroup is non-trivial for $d\geq 2$ as we mentioned above. Therefore, the problem of determining the Chow group of a Severi-Brauer variety reduces to computing the torsion subgroup in such a codimension $d$. For $d=2$, the torsion subgroup is determined when the exponent of $A$ is a prime \cite{Kar}. But, it appears to be typically difficult to determine the torsion in higher codimensions as one can see the results on other projective homogeneous varieties \cite{GMS}, \cite{Pey}, \cite{Kar96}, \cite{GZ10}, \cite{BZZ}. In particular, the torsion in the Chow group of codimension $4$ of a quadric can be infinitely generated \cite{KM}.

In this paper, we provide a new bound for the annihilator of $\CH^{d}(\SB(A))_{\tors}$ in order to estimate the torsion in higher codimensions. In general, the only known bound for the annihilator of the torsion subgroup of any codimension is the index of $A$. For $d=2$, it was shown that the torsion subgroup is annihilated by the exponent of $A$ which is given by the Rost invariant \cite{GMS}, \cite{Pey}. In particular, if in addition $A$ has the reduced sequence of length one, then the torsion subgroup is annihilated by the order of  $\Gamma^{2/3}(\SB(A))_{\tors}$ \cite{Kar}. In the present paper, we extend this to both the reduced sequence of arbitrary length and higher codimensions. Consequently, for a large class of central simple algebras, we give sharper bounds for the annihilator than the previously known ones.

This paper is organized as follows. In Section $2$, we recall some results on the gamma filtration for a Severi-Brauer variety \cite{Kar} and extend them to both higher codimension and an arbitrary reduced sequence (Theorem \ref{mainprop}). Then, we apply the main theorem to obtain the degrees of the gamma filtration where the torsion subgroup is nontrivial and show the lower bound is sharp for certain cases (Corollaries \ref{mainpropcor} and \ref{indexsquareexponentp}). We also provide an upper bound for the annihilator of the gamma filtration (Proposition \ref{propsecond}), which will be used in the last section. Sections $3$ and $4$ are devoted to applications of the main result: In Section $3$ we provide a general lower bound for the size of the torsion subgroups in both the topological filtration and the Chow group (Corollay \ref{corollaryone}) and use it to determine the indecomposablility of certain algebras (Corollary \ref{coroforindecomnew}). In the last section, we apply Proposition \ref{propsecond} to obtain an improved bound for the annihilator of the torsion in the Chow group (Corollary \ref{secondcoro}). At the end we provide several examples of computations for the annihilators (Examples \ref{exmpleforchowthm}).

We use the following notation. The base field is denoted by $F$. Given a variety $X$ (resp. an algebra $A$) over $F$ and a field extension $E/F$, we write $X_{E}$ (resp. $A_{E}$) for $X\times_{\Spec{F}} \Spec{E}$ (resp. $A\tens_{F} E$). For an abelian group $G$, we denote by $G_{\tors}$ the torsion subgroup of $G$. If $p$ is a prime integer, we write $v_{p}$ for the $p$-adic valuation on the rational numbers.

\section{Torsion in the gamma filtration of Severi-Brauer varieties}\label{secondsection}

In the present section we recall some known results on the gamma filtration of the Grothendieck ring $K_{0}$ of a Severi-Brauer variety. For a smooth projective variety $X$, the gamma filtration of $K_{0}(X)$ is given by the ideals 
\[\Gamma^{d}(X)=\langle \gamma_{d_{1}}(x_{1})\cdots \gamma_{d_{n}}(x_{n})\mid d_{1}+\cdots+d_{n}\geq d,\, x_{i}\in \Gamma^{1}(X)\rangle,\]
where $\gamma_{d_{i}}$ is the gamma operation on $K_{0}(X)$ and $\Gamma^{1}(X)=\Ker(K_{0}(X)\stackrel{\rank}\to \Z)$. We shall write $\Gamma^{d/d+1}(X)$ for the subsequent quotient $\Gamma^{d}(X)/\Gamma^{d+1}(X)$. For example, we have $\Pic(X)\simeq \Gamma^{1/2}(X)$, where $\Pic(X)$ is the Picard group of $X$. For details and the general theory of the gamma filtration, we refer the reader to \cite{Kar} and to \cite{Manin}, \cite{FL}. 

The goal of this section is to find a nontrivial torsion in the quotient $\Gamma^{d/d+1}(\SB(A))$ of a $p$-primary algebra $A$ and to provide a general lower bound for the size of the torsion subgroup $\Gamma^{d/d+1}(\SB(A))_{\tors}$ (Theorem \ref{mainprop}). In particular, we show that the lower bound is sharp if $A$ has the doubly reduced sequence of length $1$ or $A$ is of index $p^{2}$ (resp. $8$) and exponent $p$ (resp. $2$) for $p$ an odd prime (Corollaries \ref{mainpropcor} and \ref{indexsquareexponentp}). We also provide an improved bound for the annihilator of the torsion subgroup $\Gamma^{d/d+1}(\SB(A))_{\tors}$ (Proposition \ref{propsecond}).

Let $A$ be a central simple algebra of degree $\deg(A)$ over $F$, where the degree of $A$ is the square root of $\dim(A)$. Consider the restriction map 
\begin{equation}\label{Quillenres}
K_{0}(\SB(A))\to K_{0}(\P^{\deg(A)-1}_{E})\simeq \Z[x]/(x-1)^{\deg(A)},
\end{equation}
 where we identify $\SB(A)$ over a splitting field $E$ with the projective space $\P^{\deg(A)-1}_{E}$ and $x$ is the class of the tautological line bundle on $\P^{\deg(A)-1}_{E}$. Then by \cite[\S 8 Theorem 4.1]{Qui} the image of this map coincides with the sublattice with basis 
\begin{equation}\label{Quillenbasislattice}
\{ \ind(A^{\tens i})\cdot x^{i} \,|\,\, 0\leq i\leq \deg(A)-1\},
\end{equation}where $\ind(A)$ denotes the index of $A$. Hence, as observed in \cite[Corollary 3.2]{Kar} the gamma filtration on $K_{0}(\SB(A))$ depends only on $\ind(A^{\tens i})$. For instance, if $n=\ind(A)$, then we have $nx\in K_{0}(\SB(A))$ and $\gamma_{t}(nx-n)=\gamma_{t}(x-1)^{n}=(1+(x-1)t)^{n}$, where $\gamma_{t}$ is the power series associated to the gamma operations.  Hence, for any $d\geq0$ we obtain $$\gamma_{d}(nx-n)={{n}\choose{d}}(x-1)^{d}\in \Gamma^{d}(\SB(A)).$$

Given a $p$-primary algebra $A$, let $\alpha(k)=v_{p}(\ind(A^{\tens p^{k}}))$ for $0\leq k\leq v_{p}(\exp(A))$, where the exponent $\exp(A)$ of $A$ is the order of the class of $A$ in the Brauer group $\Br(F)$. The sequence $\big(\alpha(k)\big)$ is called the $p$-\emph{sequence} of $A$ of \emph{length} $v_{p}(\exp(A))$. Indeed, the elements of this sequence are distinct elements of $\{\ind(A^{\tens i})\,|\, i\geq 0\}$. This sequence is strictly decreasing such that the last term is $0$. Moreover, by \cite[Construction 2.8]{SV} or \cite[Lemma 3.10]{Kar}, any such sequence is the $p$-sequence of a division algebra of $p$-power degree.

Let $A$ be a central simple algebra of $p$-power degree and $\big(\alpha(k)\big)$ be its $p$-sequence. For any integer $\beta(k)\geq 0$, we write $p(\alpha(k),\beta(k))$ for $p^{\alpha(k)}/\gcd(\beta(k), p^{\alpha(k)})$. Then, by \cite[Proposition 4.1]{Kar}, the $d$-th gamma filtration $\Gamma^{d}(\SB(A))$ of the corresponding Severi-Brauer variety $\SB(A)$ is generated by
\begin{equation}\label{gengamma}
\prod_{k=0}^{\exp(A)}p(\alpha(k),\beta(k))(x^{p^{k}}-1)^{\beta(k)} \text{ for all } \sum_{k=0}^{\exp(A)}\beta(k)\geq d,
\end{equation}
where $x$ is the class of the tautological line bundle on $\P^{\,\deg(A)-1}_{E}$. We shall denote by
\[ f(\alpha(k), \beta(k))\]
the $k$-th entry of the product in (\ref{gengamma}). Indeed, the number of generators in (\ref{gengamma}) can be reduced as follows.

\begin{lemma}\cite[Proposition 4.10]{Kar}\label{Karlemma}
Let $A$ be a central simple algebra of $p$-power degree and $\big(\alpha(k)\big)$ be the $p$-sequence. Then, the $d$-th gamma filtration $\Gamma^{d}(\SB(A))$ of the corresponding Severi-Brauer variety $\SB(A)$ is generated by $(\ref{gengamma})$ with the condition that $\beta(k)=0$ for all $k$
satisfying $\alpha(k)=\alpha(k-1)-1$.
\end{lemma}

If the index is equal to the exponent of a central simple algebra of $p$-power degree, then by \cite[Proposition 3.3 and Corollary 3.6]{Kar}, the subsequent quotient of the gamma filtration of the corresponding Severi-Brauer variety has no torsion. Therefore, for torsion it is enough to consider the case where the exponent is less than the index of algebra. If the index of a central simple algebra $A$ of $p$-power degree is equal to the exponent of $A$, then the $p$-sequence of $A$ is $\big(\alpha(k)\big)$, where $\alpha(k)=v_{p}(\ind(A))-k$. In other words, $\ind(A)>\exp(A)$ if and only if 
\[ \exists\text{ an integer } m>0 \text{ and } 1\leq k_{1}<\ldots <k_{m}\leq v_{p}(\exp(A))\, :\,  \alpha(k_{i})\leq \alpha(k_{i}-1)-2.\]
Set $k_{0}:=0$. The subsequence $\big(\alpha(k_{i})\big)_{i=0}^{m}:=\big(\alpha(k_{0}), \ldots, \alpha(k_{m})\big)$ of the $p$-sequence $\big(\alpha(k)\big)$ is called the \emph{reduced sequence} of $A$ of \emph{length} $m$. Observe that we have
\begin{equation}\label{alphadifference}
\alpha(k_{i-1})-\alpha(k_{i})> k_{i}-k_{i-1}
\end{equation}
for all $1\leq i\leq m$.

Let $d\geq 2$ be an integer. The reduced sequence $\big(\alpha(k_{i})\big)_{i=0}^{m}$ is said to be the \emph{doubly reduced sequence} (of degree $d$) if the classes $f(\alpha(k_{i}), d)$ generate $\Gamma^{d/d+1}(X)$. In this case, we have a further reduction of the number of generators in the quotient $\Gamma^{d/d+1}(\SB(A))$. A large class of $p$-primary algebras admits such a sequence:

\begin{example}\label{firstexampleintorsiondouble}
$(i)$ Let $A$ be a $p$-primary algebra with $\ind(A)\neq \exp(A)$. If the reduced sequence $\big(\alpha(k_{i})\big)_{i=0}^{m}$ satisfies 
\begin{equation}\label{doublyreduced}
\alpha(k_{m-1})+\alpha(k_{m})+k_{m-1}\geq \alpha(0), 
\end{equation}
then by Lemma \ref{firstlemma} this sequence is doubly reduced for all $d\leq p$. Observe that 
the condition (\ref{doublyreduced}) automatically holds if the length is $1$.

In particular, any $p$-primary algebra $A$ of $\exp(A)=p$ has the doubly reduced sequence $(v_{p}(\ind(A)), 0)$ as this sequence has length $1$. Indeed, the class of $p$-primary algebras having doubly reduced sequences of length $1$ includes more algebras other than the algebras of prime exponent: for instance, any $p$-primary algebra of $\ind(A)=p^{3}$ and $\exp(A)=p^{2}$ has a doubly reduced sequence of length $1$.

\smallskip

$(ii)$ Let $A$ be a central simple algebra whose $3$-sequence is $(4, 2, 0)$. Then, the reduced sequence is also $(4, 2, 0)$, which does not satisfy the condition (\ref{doublyreduced}). However, by the proof of Lemma \ref{firstlemma} the quotient group $\Gamma^{2/3}(\SB(A))$ is generated by $f(4, 2)$, $f(2, 2)$, $f(0, 2)$ and $f(2, 1)\cdot f(0, 1)$. Moreover, by direct computation one sees that the difference $f(2, 1)\cdot f(0, 1)-3f(4, 2)$ is equal to
\begin{equation}\label{exampledifference}
9f(2, 4)+45g_{3}+32g_{4}+38g_{5}+82g_{6}+12g_{7}+3g_{8}+3g_{9},
\end{equation}
where $g_{i}=f(4, i)$ for $3\leq i\leq 9$. As each term of (\ref{exampledifference}) is contained in $\Gamma^{3}(\SB(A))$, we obtain $f(2, 1)\cdot f(0, 1)=3f(4, 2)$ in $\Gamma^{2/3}(\SB(A))$. Hence, the reduced sequence is a doubly reduced sequence of degree $2$. 
\end{example}


\begin{lemma}\label{firstlemma}
Let $A$ be a $p$-primary algebra with $\ind(A)>\exp(A)$. If the corresponding reduced sequence $\big(\alpha(k_{i})\big)_{i=0}^{m}$ satisfies $(\ref{doublyreduced})$, then for any $d\leq p$ the classes of $f(\alpha(k_{i}), d)$ generate $\Gamma^{d/d+1}(X)$.
\end{lemma}
\begin{proof}
Let $X=\SB(A)$ and 
\begin{equation}\label{lemma1gen}
f_{j}:=\prod_{i=j}^{m} f(\alpha(k_{i}), \beta(k_{i})), \,\,\, \sum_{i=j}^{m}\beta(k_{i})=d
\end{equation}
for $j=0, 1$. Then, by Lemma \ref{Karlemma}, the class of $f_{0}$ generates $\Gamma^{d/d+1}(X)$. We show that any $f_{0}$ is generated by $f(\alpha(k_{i}), d)$.

We first assume that $1\leq \beta(0)\leq d-1$. Then, by the assumption that $d\leq p$, we have $p(\alpha(0), \beta(0))=p^{\alpha(0)}$ and $p(\alpha(0),d)=p^{\alpha(0)}$ (resp. $p^{\alpha(0)-1}$) if $d<p$ (resp. $d=p$). Let $y=x-1$, $g(j)=\sum_{l=1}^{j}{{j}\choose{l}}y^{l}$ for any integer $j\geq 1$, and
\[e_{m}=p(\alpha(0), \beta(0))\cdot p(\alpha(0), d)^{-1}\prod_{i=1}^{m}p(\alpha(k_{i}), \beta(k_{i}))\cdot p^{\sum_{i=1}^{m}k_{i}\beta(k_{i})}.\]
Then, the term $f_{0}-e_{m}\cdot f(\alpha(0), d)$ can be written as
\begin{equation}\label{trivialeq}
p^{\alpha(0)}y^{\beta(0)}\prod_{i=1}^{m}\big(p(\alpha(k_{i}), \beta(k_{i}))\cdot g(p^{k_{i}})^{\beta(k_{i})}-p(\alpha(k_{i}), \beta(k_{i}))\cdot p^{\sum_{i=1}^{m}k_{i}\beta(k_{i})}y^{d-\beta(0)}\big).
\end{equation}
As $\sum_{i=1}^{m}\beta(k_{i})=d-\beta(0)$, there is no term of degree less or equal to $d$ in (\ref{trivialeq}). Moreover, each term of a power of $y$ in (\ref{trivialeq}) is a multiple of $p^{\alpha(0)}$. Therefore, we obtain $f_{0}-e_{m}\cdot f(\alpha(0), d)\in \Gamma^{d+1}(X)$. Thus, we may assume that $\beta(0)=0$ and $f_{0}=f_{1}$.

If $\beta(k_{i})=d$ for some $i$, then $f_{1}=f(\alpha(k_{i}), d)$, so we may assume that $\beta(k_{i})\leq d-1$ for all $1\leq i\leq m$. In particular, the result holds for $m=1$. Now we assume that $m\geq 2$ and the corresponding reduced sequence $\big(\alpha(k_{i})\big)_{i=0}^{m}$ satisfies (\ref{doublyreduced}). We show that any $f_{1}$ in (\ref{lemma1gen}) is a multiple of $f(\alpha(0), d)$ modulo $\Gamma^{d+1}(X)$.

Suppose that $1\leq \beta(k_{i})\leq d-1$ for all $1\leq i\leq m$. Then, it follows from $d\leq p$ that we have $p(\alpha(k_{i}), \beta(k_{i}))=p^{\alpha(k_{i})}$. Therefore, by (\ref{doublyreduced}), we obtain $e_{m}\geq 1$. Consider
\begin{equation}\label{inductionfirst}
f_{1}-e_{m}\cdot f(\alpha(0), d)=p^{\sum_{i=1}^{m}\alpha(k_{i})}\big(\prod_{i=1}^{m}g(p^{k_{i}})^{\beta(k_{i})} -p^{\sum_{i=1}^{m}k_{i}\beta(k_{i})}y^{d}\big)
\end{equation}
and the coefficient
\[c_{j}:=\sum \Big(\prod_{i=1}^{m} {p^{k_{i}}\choose{a_{1}^{\!(i)}}}\cdots {p^{k_{i}}\choose{a_{\beta(k_{i})}^{(i)}}}\Big)\]
of $y^{j}$ in $(f_{1}-e_{m}\cdot f(\alpha(0), d))p^{-\sum_{i=1}^{m}\alpha(k_{i})}$ for $d+1\leq j\leq \sum_{i=1}^{m}k_{i}\beta(k_{i})$, where the sum ranges over all partitions of
$\sum_{i=1}^{m}\sum_{l=1}^{\beta(k_{i})} a_{l}^{\!(i)}=j$ with $1\leq a_{l}^{\!(i)}\leq p^{k_{i}}$. As there exists at least one $a_{l}^{\!(i)}$ such that $v_{p}(a_{l}^{\!(i)})\leq v_{p}(j)$ for each summand of $c_{j}$, we have 
\begin{equation}\label{mequaltwo}
v_{p}(c_{j})\geq k_{1}-v_{p}(j)
\end{equation}
for all $d+1\leq j\leq \sum_{i=1}^{m}k_{i}\beta(k_{i})$. Since $v_{p}(p(\alpha(0), j))=\alpha(0)-v_{p}(j)$, it follows from (\ref{mequaltwo}) and (\ref{doublyreduced}) that $$v_{p}(p(\alpha(0), j))\leq v_{p}(c_{j})+\sum_{i=1}^{m}\alpha(k_{i}),$$ which means that each term of a power of $y$ in (\ref{inductionfirst}) is contained in $\Gamma^{d+1}(X)$. This finishes the proof of the case $1\leq \beta(k_{i})\leq d-1$. In particular, the result holds for $m=2$.

We use induction on $m\geq 2$ to prove that $f_{1}$ is a multiple of $f(\alpha(0), d)$ modulo $\Gamma^{d+1}(X)$. Assume that the result holds for $m-1$. By the argument of the preceding paragraph, it suffices to consider the case where $\beta(k_{i})=0$ for each $1\leq i\leq m$. By the induction hypothesis, the result holds for each case of $\beta(k_{i})=0$, $1\leq i\leq m-2$, provided that the condition (\ref{doublyreduced}). Again by the induction hypothesis, the result holds for $\beta(k_{m-1})=0$ (resp. $\beta(k_{m})=0$) if 
\begin{equation}\label{lastinequallemma}
\alpha(k_{m-2})+\alpha(k_{m})+k_{m-2}\geq \alpha(0) \text{ (resp. } \alpha(k_{m-2})+\alpha(k_{m-1})+k_{m-2}\geq \alpha(0)).
\end{equation}
But by (\ref{alphadifference}) and (\ref{doublyreduced}), the inequalities in (\ref{lastinequallemma}) hold, which finishes the proof.\end{proof}


We shall need the following lemma, which generalize a proof of \cite[Proposition 4.7]{Kar}.

\begin{lemma}\label{congruentlem}
Let $A$ be a $p$-primary algebra with the corresponding reduced sequence $\big(\alpha(k_{i})\big)_{i=0}^{m}$ and let $y_{i}$ be the corresponding tautological line bundle $x$ to the $p^{k_{i}}$-power. Set $\epsilon(i)=\alpha(k_{i-1})+k_{i-1}-k_{i}$. Then, for any $1\leq i\leq m$ and $d\geq 1$ satisfying $\epsilon(i)-[\log_{p} d]\geq 1$, and for all $f\in \Gamma^{d+1}(\SB(A))$  there exists a polynomial $\phi$ in $y_{i}$ such that
\[f\equiv (y_{i}-1)^{d+1}\phi(y_{i}) \mod p^{\epsilon(i)-[\log_{p} d]}.\]
\end{lemma}

\begin{proof}
Let $\epsilon(i)=\alpha(k_{i-1})+k_{i-1}-k_{i}$ and $y_{i}=x^{p^{k_{i}}}$ for $1\leq i\leq m$. Fix $i$ and a positive integer $d$ such that $\epsilon(i)-[\log_{p} d]\geq 1$. Set
\[f_{i}=\prod_{j=i}^{m} p(\alpha(k_{j}), \beta(k_{j}))(y_{i}^{p^{k_{j}-k_{i}}}-1)^{\beta(k_{j})}.\]
Then, a generator $f_{0}$ of $\Gamma^{d+1}(X)$ can be written as
\[f_{0}=f_{i}\cdot \prod_{j=0}^{i-1} p(\alpha(k_{j}), \beta(k_{j}))(y_{j}-1)^{\beta(k_{j})}, \text{ where } \sum_{j=0}^{m}\beta(k_{j})\geq d+1 \text{ and }y_{0}=x.\]
If $\sum_{j=0}^{i-1} \beta(k_{j})=0$, then $f_{0}$ is a multiple of $f_{i}$ and $f_{i}$ is a multiple of $(y_{i}-1)^{d+1}$, thus we have $f_{0}=(y_{i}-1)^{d+1}\phi(y_{i})$ for some polynomial $\phi$ in $y_{i}$. Therefore, we may assume that $\beta(k_{j})\neq 0$ for all $0\leq j\leq i-1$ since we can exclude the term of $\beta(k_{j})=0$ in $f_{0}$ and rearrange the index. Moreover, if one of the $v_{p}(p(\alpha(k_{j}), \beta(k_{j})))$, $j=0,\ldots, i-1$ is bigger or equal to $\epsilon(i)-v_{p}(d)$, then we are done. Therefore, we may assume that
\begin{equation}\label{conditioninlem}
\epsilon(i)-[\log_{p} d]>v_{p}(p(\alpha(k_{j}), \beta(k_{j})))
\end{equation}
for all $j$.

By (\ref{alphadifference}), we have $\alpha(k_{j})+k_{j}\geq \alpha(k_{i-1})+k_{i-1}$ for all $0\leq j\leq i-1$. Hence, it follows from $v_{p}(\beta(k_{j}))+v_{p}(p(\alpha(k_{j}),\beta(k_{j})))=\alpha(k_{j})$ that
\begin{equation}\label{condonelem}
v_{p}(\beta(k_{j}))-k_{i}+k_{j}+v_{p}(p(\alpha(k_{j}), \beta(k_{j})))\geq \epsilon(i)
\end{equation}
for all $0\leq j\leq i-1$. The inequality (\ref{conditioninlem}) implies that
\begin{equation}\label{condtwolem}
v_{p}(\beta(k_{j}))-k_{i}+k_{j}\geq [\log_{p} d]+1.
\end{equation}

Observe that, by \cite[Lemma 4.8]{Kar}, we have 
\[(y_{j}-1)^{p^{r}}\equiv (y_{j}^{p}-1)^{p^{r-1}}\]
modulo $p^{r}$ for any $r\geq 1$.
By applying this formula, for any $0\leq j\leq i-1$ we obtain
\[p(\alpha(k_{j}), \beta(k_{j}))(y_{j}-1)^{\beta(k_{j})}\equiv p(\alpha(k_{j}), \beta(k_{j}))(y_{i}-1)^{\beta(k_{j})\cdot p^{k_{j}-k_{i}}}\]
modulo $p$ to the ${v_{p}(\beta(k_{j}))-k_{i}+k_{j}+v_{p}(p(\alpha(k_{j}), \beta(k_{j})))}+1$-power. Since we have $[\log_{p} d]+1\geq \log_{p}(d+1)$, the result follows from (\ref{condonelem}) and (\ref{condtwolem}).
\end{proof}

We are ready to prove the main result of this section.


\begin{theorem}\label{mainprop}
Let $A$ be a $p$-primary algebra with $\exp(A)< \ind(A)$, $\big(\alpha(k_{i})\big)_{i=0}^{m}$ the corresponding reduced sequence, $\epsilon(i)=\alpha(k_{i-1})+k_{i-1}-k_{i}$, and $X=\SB(A)$. Then, for any $i\geq 1$ and $1< d< p^{\epsilon(i)}$, the torsion subgroup $\Gamma^{d/d+1}(X)_{\tors}$ contains at least $p^{\lambda(i, d)}$ elements, where
\[\lambda(i, d)= \epsilon(i)+v_{p}(d)-[\log_{p}{d}]-\!
\begin{cases}
(\alpha(0)-dk_{i}) &\!\!\! \text{ if } \alpha(0)\geq dk_{i}+\!\max\{\alpha(k_{i}), v_{p}(d)\},\\
\max\{\alpha(k_{i}), v_{p}(d)\} &\!\!\! \text{ otherwise.}
\end{cases}
\]
In particular, $|\Gamma^{d/d+1}(X)_{\tors}|\geq \max\{p^{\lambda(i, d)}\,|\, 1\leq i\leq m,\,\, 1< d< p^{\epsilon(i)}\}$.
\end{theorem}
\begin{proof}
Let $A$ be a central simple algebra of $p$-power degree such that $\exp(A)<\ind(A)$. Then, $A$ has the reduced sequence $\big(\alpha(k_{i})\big)_{i=0}^{m}$ for some $m\geq 1$. Consider the elements $f(\alpha(0), d), \ldots, f(\alpha(k_{m}), d)$ of $\Gamma^{d}(X)$, where $X=\SB(A)$. Set $$e_{i}=p(\alpha(0), d)\cdot p(\alpha(k_{i}), d)^{-1}\cdot p^{-dk_{i}}$$ for $1\leq i\leq m$. Then, we have $e_{i}\geq 1$ if 
\begin{equation}\label{eipositive}
\alpha(0)+\min\{ \alpha(k_{i}), v_{p}(d)\}\geq v_{p}(d)+\alpha(k_{i})+dk_{i}
\end{equation}
and $e_{i}^{-1}\geq 1$ if
\begin{equation}\label{eiinvpositive}
v_{p}(d)+\alpha(k_{i})+dk_{i}\geq \alpha(0)+\min\{ \alpha(k_{i}), v_{p}(d)\}.
\end{equation}
Consider the following elements of $\Gamma^{d}(X)$
\begin{equation*}
t_{i}=f(\alpha(0), d)-e_{i}\cdot f(\alpha(k_{i}), d)\,\,\,\, \text{ for } (\ref{eipositive})
\end{equation*}
and $t'_{i}=e_{i}^{-1}\cdot t_{i}$ for (\ref{eiinvpositive}). Then, by direct computation, one has
\begin{align}
t_{i}&=p^{\alpha(0)-v_{p}(d)}y^{d}-p^{\alpha(0)-dk_{i}-v_{p}(d)}(y_{i}-1)^{d} \,\,\, \text{ and }\label{tjeq} \\ 
t'_{i}&=p^{\alpha(k_{i})+dk_{i}-\min\{\alpha(k_{i}), v_{p}(d)\}}y^{d}-p^{\alpha(k_{i})-\min\{\alpha(k_{i}), v_{p}(d)\}}(y_{i}-1)^{d} \label{tprimejeq} 
\end{align} 
where $y=x-1$ and $y_{i}=x^{p^{k_{i}}}$.

Let $g_{i}(y)=\big({{p^{k_{i}}}\choose{1}}y^{1}+{{p^{k_{i}}}\choose{2}}y^{2}+\cdots +{{p^{k_{i}}}\choose{p^{k_{i}}}}y^{p^{k_{i}}}\big)^{d}$. Consider the coefficient
\[c_{j}:=\sum {p^{k_{i}}\choose{a_{1}}}\cdots {p^{k_{i}}\choose{a_{d}}}\]
of $y^{j}$ in $g_{i}(y)$ for $d+1\leq j\leq dp^{k_{i}}$, where the sum ranges over all partitions of
$\sum_{l=1}^{d}a_{l}=j$ with $1\leq a_{l}\leq p^{k_{i}}$. We claim that 
\begin{equation}\label{covadic2}
k_{i}-v_{p}(j)+v_{p}(d)\leq v_{p}(c_{j}).
\end{equation}
Let $c'_{j}$ be the summand of $c_{j}$ such that the sum ranges over all partitions of $a_{1}=\cdots=a_{d}$. Then, for each partition $\sum_{l=1}^{d}a_{l}=j$ of $c'_{j}$ we have $v_{p}(j)=v_{p}(d)+v_{p}(a_{l})$ for all $1\leq l\leq p$. Hence, $k_{i}-v_{p}(j)+v_{p}(d)=k_{i}-v_{p}(a_{l})\leq v_{p}(c'_{j})$. Set $c''_{j}=c_{j}-c'_{j}$. For each partition $\sum_{l=1}^{d}a_{l}=j$ of $c''_{j}$ such that $a_{1}\geq\cdots \geq a_{d}$ there are a multiple of $p^{v_{p}(d)}$ permutations of this partition. Moreover, the collection of such permutations form all partitions of $c''_{j}$. Hence, for each partition $\sum_{l=1}^{d}a_{l}=j$ of $c''_{j}$ there exists at least one $a_{l}$ in the partition such that $v_{p}(j)\geq v_{p}(a_{l})$ and this appears at least a multiple of $p^{v_{p}(d)}$ times in $c''_{j}$. Therefore, $(k_{i}-v_{p}(a_{l}))+v_{p}(d)\leq v_{p}(c''_{j})$. As $\min\{v_{p}(c'_{j}), v_{p}(c''_{j})\}\leq v_{p}(c_{j})$, we obtain (\ref{covadic2}).

Let $M_{i}=p^{(d-1)k_{i}}$ and $M'_{i}=p^{\alpha(0)-\max\{\alpha(k_{i}), v_{p}(d)\}-k_{i}}$ for $1\leq i\leq m$. Then, we have $M_{i}t_{i}=M'_{i}t'_{i}$. To show $M_{i}t_{i}\in \Gamma^{d+1}(X)$, it is enough to verify that each coefficient of $y^{j}$ in $M_{i}t_{i}$ is a multiple of $p(\alpha(0), j)$ for $d+1\leq j\leq dp^{k_{i}}$. It follows from (\ref{tjeq}) that the coefficient of $y^{j}$ in $M_{i}t_{i}$ is $p^{\alpha(0)-k_{i}-v_{p}(d)}c_{j}$. Hence, by (\ref{covadic2}), we have 
\begin{equation*}
v_{p}(p(\alpha(0),j))\leq \alpha(0)-k_{i}-v_{p}(d)+v_{p}(c_j),
\end{equation*}
which implies that 
\begin{equation}\label{mnannihilators}
M_{i}t_{i}=M'_{i}t'_{i}\in \Gamma^{d+1}(X).
\end{equation}

Let $\epsilon(i)=\alpha(k_{i-1})+k_{i-1}-k_{i}$. Assume that an integer $1<d<p^{\epsilon(i)}$ satisfies 
\begin{equation}\label{ticonditioninthm}
\alpha(k_{i})-\min\{\alpha(k_{i}), v_{p}(d)\}\leq \alpha(0)-dk_{i}-v_{p}(d)< \epsilon(i)-[\log_{p}{d}]
\end{equation}
for some $i$. Let $\lambda(i, d)=\epsilon(i)+v_{p}(d)+dk_{i}-[\log_{p}(d)]-\alpha(0)$. Then, we have $0<\lambda(i, d)\leq M_{i}$. We show that the element $p^{\lambda(i,d)-1}t_{i}$ is not contained in $\Gamma^{d+1}(X)$. Assume the contrary. Then, it follows from Lemma \ref{congruentlem} and (\ref{tjeq}) that
\begin{equation}\label{tcontainedingammad}
-p^{\epsilon(i)-[\log_{p}{d}]-1}(y_{i}-1)^{d}=(y_{i}-1)^{d+1}\phi(y_{i})+p^{\epsilon(i)-[\log_{p}{d}]}\psi(y_{i})
\end{equation}
for some polynomial $\psi(y_{i})$. Thus, the equation (\ref{tcontainedingammad}) becomes $1=(y_{i}-1)\rho(y_{i})+p\mu(y_{i})$ for some polynomials $\rho(y_{i})$ and $\mu(y_{i})$, which gives a contradiction. Hence, we obtain $p^{\lambda(i, d)-1}t_{i}\notin \Gamma^{d+1}(X)$. Therefore, by (\ref{mnannihilators}) the class of $t_{i}$ gives a torsion element of order  greater than or equal to $p^{\lambda(i, d)}$.

Now, we assume that an integer $d$ satisfies $1< d< p^{\epsilon(i)}$ and 
\begin{equation}\label{tiprimeconditioninthm}
\alpha(0)-dk_{i}-v_{p}(d)\leq \alpha(k_{i})-\min\{\alpha(k_{i}), v_{p}(d)\}< \epsilon(i)-[\log_{p}{d}]
\end{equation}
for some $i$. Let $\lambda'(i, d)=\epsilon(i)+v_{p}(d)-\max\{\alpha(k_{i}), v_{p}(d)\}-[\log_{p}{d}]$. Then, we have $0<\lambda'(i, d)\leq M'_{i}$. If $p^{\lambda'(i, d)-1}t'_{i}\in \Gamma^{d+1}(X)$, then, again by Lemma \ref{congruentlem} and (\ref{tprimejeq}) we obtain the same equation (\ref{tcontainedingammad}). By the same argument, we have $p^{\lambda'(i, d)-1}t'_{i}\notin \Gamma^{d+1}(X)$. Hence, by (\ref{mnannihilators}) the class of $t'_{i}$ gives a torsion element of order bigger than or equal to $p^{\lambda'(i, d)}$.\end{proof}

We apply Theorem \ref{mainprop} to particular cases:


\begin{corollary}\label{mainpropcor}
Let $A$ be a $p$-primary algebra with $\exp(A)< \ind(A)$ and $\big(\alpha(k_{i})\big)_{i=0}^{m}$ the corresponding reduced sequence. Then, for any $1< d <p^{\epsilon(1)}$ of the form $jp^{n}$ with $1\leq j< p$ and $0\leq n< \epsilon(1)$, the torsion subgroup $\Gamma^{d/d+1}(\SB(A))_{\tors}$ is nontrivial and
$$|\Gamma^{d/d+1}(\SB(A))_{\tors}|\geq \max\{ p^{\lambda(i, d)}\,|\, 1\leq i\leq m, \,\, 1<d=jp^{n}<p^{\epsilon(i)}\},$$
where $\lambda(i, d)=\min\{\epsilon(i)+dk_{i}-\alpha(0),\,\, \epsilon(i)-\max\{\alpha(k_{i}), n\}\}$. In particular, if $A$ has the doubly reduced sequence of length one, then for any such $d$, the group $\Gamma^{d/d+1}(\SB(A))_{\tors}$ is cyclic of order $p^{\lambda(1, d)}$.\end{corollary}

\begin{proof}
Let $1< d< p^{\epsilon(1)}$ be an integer of the form $jp^{n}$ with $1\leq j<p$ and $0\leq n< \epsilon(1)$. Then, we have $\lambda(1, d)=\min\{{(d-1)k_{1}}, \epsilon(1)-\max\{\alpha(k_{1}), n\}\}\geq 1$, thus the first result follows from Theorem \ref{mainprop}. Assume that $A$ has the doubly reduced sequence of $m=1$. Then, the classes of $f(\alpha(0), d)$ and $f(\alpha(k_{1}), d)$ generate $\Gamma^{d/d+1}(\SB(A))$. Hence, by (\ref{mnannihilators}) we have $\Gamma^{d/d+1}(\SB(A))_{\tors}=\Z/p^{\lambda(1, d)}\Z$ which is generated by $t_{1}$ and $t'_{1}$, respectively.\end{proof}

\begin{corollary}\label{indexsquareexponentp}
Let $A$ be a $p$-primary algebra. If the reduced sequence of $A$ satisfies $\alpha(k_{1})=0$, then for any $2\leq d< p^{\epsilon(1)}$, the group $\Gamma^{d/d+1}(\SB(A))_{\tors}$ is nontrivial and \[ |\Gamma^{d/d+1}(\SB(A))_{\tors}|\geq \min\{ p^{(d-1)k_{1}+v_{p}(d)-[\log_{p}{d}]}, p^{\epsilon(1)-[\log_{p}{d}]} \}.\] 
In particular, if $A$ is a division algebra of index $p^{2}$ and exponent $p$ for $p$ an odd prime $($resp. index $8$ and exponent $2$$)$, then for any $2\leq d\leq p-1$ $($resp. $2\leq d\leq 3$$)$
\[\bigoplus_{d=0}^{p^{2}-1} \Gamma^{d/d+1}(\SB(A))_{\tors}=(\Z/p\Z)^{\oplus p-2}\, \text{ and }\, \Gamma^{d/d+1}(\SB(A))_{\tors}=\Z/p\Z\]
$($resp. $\bigoplus_{d=0}^{7} \Gamma^{d/d+1}(\SB(A))_{\tors}=(\Z/2\Z)^{\oplus 2}\, \text{ and }\, \Gamma^{d/d+1}(\SB(A))_{\tors}=\Z/2\Z )$.\end{corollary}

\begin{proof}
If $\alpha(k_{1})=0$, then $\lambda(1, d)=\min\{(d-1)k_{1}+v_{p}(d)-[\log_{p}{d}], \epsilon(1)-[\log_{p}{d}]\}\geq 1$, thus the first statement follows from Theorem \ref{mainprop}. 

Let $A$ be a division algebra of index $p^{2}$ and exponent $p$ for $p$ an odd prime, $X=\SB(A)$ and $E$ a splitting field of $X$. We first measure the size of the torsion subgroups by using the following formula \cite[Proposition 2]{Kar95}
\begin{equation}\label{usekarfor}
|\bigoplus_{d=0}^{p^{2}-1}\Gamma^{d/d+1}(X)_{\tors}|\cdot|K(X_{E})/K(X)|=\prod_{d=1}^{p^{2}-1}|\Gamma^{d/d+1}(X_{E})/\Im(\res^{d/d+1})|,
\end{equation}
where $\res^{d/d+1}: \Gamma^{d/d+1}(X)\to \Gamma^{d/d+1}(X_{E})$ is the restriction map. 

It follows from (\ref{Quillenbasislattice}) that $|K(X_{E})/K(X)|=p^{2p(p-1)}$. By (\ref{gengamma}), we have 
\begin{equation*}
\sum_{n=1}^{p}{{p}\choose{n}}y^{n}\in \Gamma^{1}(X)\, \text{ and }\, p^{2}\cdot y^{d}/\gcd(d, p^{2})\in \Gamma^{d}(X) 
\end{equation*}
for all $2\leq d\leq p^{2}-1$. Hence, by (\ref{usekarfor}) we obtain $|\bigoplus_{d=0}^{p^{2}-1}\Gamma^{d/d+1}(X)_{\tors}|\leq p^{p-2}$. As $(\alpha(0), \alpha(1))=(2, 0)$ is the reduced sequence of $A$, we have $\lambda(1, d)=1$ for all $2\leq d\leq p-1$, thus the result follows from the first statement of the corollary.

Now we assume that $A$ is a division algebra of index $8$ and exponent $2$. It follows from (\ref{gengamma}) that \[2y\in \Gamma^{1}(X),\,\, 2^{2}y^{5}\in \Gamma^{5}(X), \text{ and } 2^{3}y^{d}/\gcd(d, 2^{3})\in \Gamma^{d}(X)\]
for $2\leq d\neq 5\leq 7$. As $|K(X_{E})/K(X)|=2^{12}$, we have $|\bigoplus_{d=0}^{7}\Gamma^{d/d+1}(X)_{\tors}|\leq 2^{2}$. Since $\lambda(1, d)=1$ for $d=2, 3$, the result follows.\end{proof}


We provide a new upper bound for the annihilator of the torsion subgroup of $\Gamma^{d/d+1}(\SB(A))$.

\begin{proposition}\label{propsecond}
Let $A$ be a $p$-primary algebra having the doubly reduced sequence $\big(\alpha(k_{i})\big)_{i=0}^{m}$ of degree $d\geq 2$ of the form $jp^{n}$ with $1\leq j<p$ and $0\leq n<\epsilon(m)$. Then, the torsion subgroup $\Gamma^{d/d+1}(\SB(A))_{\tors}$ is annihilated by
$$p^{\alpha(0)-\max\{\alpha(k_{m}), v_{p}(d)\}-k_{m}} $$ 
if $dk_{i}+\max\{\alpha(k_{i}), v_{p}(d)\}\geq \alpha(0)$ for all $i$ and is annihilated by 
$$\max\{p^{\alpha(0)-\max\{\alpha(k_{m}), v_{p}(d)\}-k_{m}}, p^{(d-1)k_{l}}\}\quad (\text{resp.} \max\{p^{\alpha(0)-\alpha(k_{m-1})-k_{m-1}}, p^{(d-1)k_{m}}\})$$
if $dk_{l}+\max\{\alpha(k_{l}), v_{p}(d)\}\leq \alpha(0)\leq \min\limits_{i\neq l}\{ dk_{i}+\max\{\alpha(k_{i}), v_{p}(d)\}\}$ 
for some $1\leq l<m$ $($resp. $l=m)$.
\end{proposition}

\begin{proof}
Let $A$ be a central simple algebra of $p$-power degree which admits the doubly reduced sequence $\big(\alpha(k_{i})\big)_{i=0}^{m}$ of degree $d$ of the form $jp^{n}$ with $1\leq j< p$ and $0\leq n< \epsilon(m)$, i.e., the classes of $f(\alpha(0), d), \ldots, f(\alpha(k_{m}), d)$ generate $\Gamma^{d/d+1}(X)$, where $X=\SB(A)$. Consider the elements $t_{i}$, $t'_{i}$, $\epsilon(i)$, $M_{i}$, and $M'_{i}$ as defined in the proof of Theorem \ref{mainprop}. We split into two cases.

{\it Case 1 :} $dk_{i}+\max\{\alpha(k_{i}), v_{p}(d)\}\geq \alpha(0)$ for all $1\leq i\leq m$. Then, the degree $d$ satisfies (\ref{tiprimeconditioninthm}) for any $i$, thus by the proof of Theorem \ref{mainprop} the class of $t'_{i}$ is a torsion element in $\Gamma^{d/d+1}(X)$. As the order of the class of $f(\alpha(k_{i}), d)$ is not finite, any torsion element in $\Gamma^{d/d+1}(X)$ is a linear combination of the classes of $t'_{1},\ldots, t'_{m}$. Since each class $t'_{i}$ is annihilated by $M'_{i}$ and $M'_{1}<\cdots <M'_{m}$, the torsion subgroup of $\Gamma^{d/d+1}(X)$ is annihilated by $M'_{m}$. 

\smallskip

{\it Case 2:} $dk_{l}+\max\{\alpha(k_{l}), v_{p}(d)\}\leq \alpha(0)$ for some $1\leq l\leq m$ and $\alpha(0)\leq dk_{i}+\max\{\alpha(k_{i}), v_{p}(d)\}$ for all $i\neq l$. Then, the degree $d$ satisfies (\ref{ticonditioninthm}) for $i=l$. Hence, the class of $t_{l}$ is a torsion element in $\Gamma^{d/d+1}(X)$. As the class of $t'_{i}$ is a torsion element in $\Gamma^{d/d+1}(X)$, any torsion element in $\Gamma^{d/d+1}(X)$ is a linear combination of the classes of $t_{l}$ and $t'_{i}$. As the class $t_{l}$ is annihilated by $M_{l}$ and the class $t'_{i}$ is annihilated by $M'_{i}$, the torsion subgroup of $\Gamma^{d/d+1}(X)$ is annihilated by $\max\{M_{l}, M'_{i}\}$. \end{proof}

\begin{example}\label{examforprop}
$(i)$ Let $A$ be a central simple algebra of $\ind(A)=5^{8}$ and $\exp(A)=5^{5}$ whose $5$-sequence is $(8,7,4,2,1,0)$. As $(8, 4, 2)$ is the reduced sequence of $A$, by Corollary \ref{mainpropcor}, the group $\Gamma^{d/d+1}(\SB(A))_{\tors}$ is nontrivial for any $d=j\cdot 5^{n}\geq $ with $1\leq j\leq 4$ and $0\leq n\leq 5$. Moreover, it follows from Proposition \ref{propsecond}  that for $d\leq 5$, the group $\Gamma^{d/d+1}(\SB(A))_{\tors}$ is annihilated by $5^{3}$.

\smallskip

$(ii)$ Consider the central simple algebra of $A$ of $\ind(A)=3^{4}$ such that its $3$-sequence is $(4, 2, 0)$. By Corollary \ref{mainpropcor}, the torsion subgroup $\Gamma^{2/3}(\SB(A))_{\tors}$ is nontrivial. As we have shown in Example \ref{firstexampleintorsiondouble} $(ii)$, the algebra $A$ admits the doubly reduced sequence. Hence, by Proposition \ref{propsecond} the torsion subgroup $\Gamma^{2/3}(\SB(A))_{\tors}$ is annihilated by $3^2$.

\end{example}

\section{Torsion in the higher Chow groups of Severi-Brauer varieties} 

In this section we apply Theorem \ref{mainprop}, Corollaries \ref{mainpropcor} and \ref{indexsquareexponentp} to provide a torsion in a higher quotient of the topological filtration and the Chow groups of a generic variety associated to a central simple algebra (see Corollary \ref{corollaryone}). As an application, we show indecomposability of certain algebras in Corollary \ref{coroforindecomnew}.

Consider the topological filtration on the Grothendieck ring $K(X)$ of a smooth projective variety $X$
\[K(X)=T^{0}(X)\supset T^{1}(X)\supset \ldots \supset T^{d}(X)\supset \ldots\]
given by the ideal $T^{d}(X)$ generated by the class $[\mathcal{O}_{Y}]$ of the structure sheaf of a closed subvariety $Y$ of codimension at least $d$. Note that the gamma filtration $\Gamma^{d}(X)$ is contained in the topological filtration $T^{d}(X)$. There is a natural surjective map from the Chow group $\CH^{d}(X)$ of cycles modulo the rational equivalence relation to the quotient of the topological filtration:
\begin{equation}\label{naturaltopChow}
\CH^{d}(X)\twoheadrightarrow T^{d}(X)/T^{d+1}(X),
\end{equation}
given by $[Y]\mapsto [\mathcal{O}_{Y}]$. We simply write $T^{d/d+1}(X)$ for the quotient $T^{d}(X)/T^{d+1}(X)$. By the Riemann-Roch theorem, there is a reverse map $T^{d/d+1}(X)\to \CH^{d}(X)$ induced by the $d$th Chern class and their composition 
\begin{equation}\label{compositecp}
\CH^{d}(X)\twoheadrightarrow T^{d/d+1}(X)\to \CH^{d}(X),
\end{equation}
is the multiplication by $(-1)^{d-1}(d-1)!$ \cite[Ex.15.3.6]{Fu}.

We now recall the notions of \emph{generic variety} and \emph{generic algebra} from \cite[Definition 3.12]{Kar}: the notion of generic algebra was originally introduced by Schofield and Van den Bergh \cite{SV}, but here we use a variation of it. Let $A$ be a $p$-primary algebra and let $\big(\alpha(k)\big)$ be the corresponding $p$-sequence. Choose any division algebra $B$ of $\ind(B)=\exp(B)=\ind(A)$ over a field $K$ and consider the function field $\bar{K}$ of the product of generalized Severi-Brauer varieties $\SB(p^{\alpha(k)}, B^{\tens p^{k}})$ over all $k\geq 1$. Then, by the index reduction formula, the algebra $\bar{A}:=B_{\bar{K}}$, called a \emph{generic algebra}, has the same $p$-sequence of $A$. The corresponding Severi-Brauer variety $\SB(\bar{A})$ is called a \emph{generic variety} associated to $A$ and is denoted by $\bar{X}$. Moreover, by \cite[Theorem 3.7]{Kar}, we have
\begin{equation*}
\Gamma^{d/d+1}(\bar{X})=T^{d/d+1}(\bar{X}).
\end{equation*}  

Applying Theorem \ref{mainprop} and Corollaries \ref{mainpropcor} and \ref{indexsquareexponentp} to a generic variety together with (\ref{naturaltopChow}), we obtain

\begin{corollary}\label{corollaryone}
Let $A$ be a $p$-primary algebra with $\exp(A)< \ind(A)$, $\big(\alpha(k_{i})\big)_{i=0}^{m}$ the corresponding reduced sequence and $\bar{X}$ a generic variety associated to $A$. Then,
\[|\CH^{d}(\bar{X})_{\tors}|\geq |T^{d/d+1}(\bar{X})_{\tors}|\geq \max\{p^{\lambda(i, d)}\,|\, 1\leq i\leq m,\,\, 1< d< p^{\epsilon(i)}\}.\]
In particular, for every $d\geq 2$ of the form $jp^{n}$ with $1\leq j< p$ and $0\leq n< \epsilon(1)$, both the Chow group $\CH^{d}(\bar{X})$ and the topological filtration $T^{d/d+1}(\bar{X})$ contain $p$-torsion. Moreover, if $A$ has the doubly reduced sequence of length one, then for any such $d$, the torsion subgroup $\CH^{d}(\bar{X})_{\tors}$ $($resp. $T^{d/d+1}(\bar{X})_{\tors}$$)$ contains $($resp. coincides with$)$ a cyclic group of order $p^{\lambda(1, d)}$. Moreover, if $\alpha(k_{1})=0$, then for every $2\leq d< p^{\epsilon(1)}$ both groups $\CH^{d}(\bar{X})_{\tors}$ and $T^{d/d+1}(\bar{X})_{\tors}$ are nontrivial.
\end{corollary}

\begin{example}\label{examplefordecomp}
$(i)$ Let $p$ be an odd prime, $A$ a division algebra of index $p^{2}$ and exponent $p$, and $X_{1}=\SB(A)$. Then by the proof of Corollary \ref{indexsquareexponentp} together with \cite[Proposition 2]{Kar95} we have $\big|\bigoplus_{d=0}^{p^{2}-1} T^{d/d+1}(X_{1})_{\tors}\big|\leq p^{p-2}.$ 

On the other hand, it follows by Corollary \ref{corollaryone} that for any $2\leq d\leq p-1$ and generic variety $\bar{X_{1}}$ associated to $A$ we have $\Z/p\Z\subseteq  T^{d/d+1}(\bar{X_{1}})_{\tors}$. Therefore, the variety $\bar{X_{1}}$ gives an example of having the maximal torsion subgroup $(\Z/p\Z)^{\oplus p-2}$ in the topological filtration, i.e.,
\begin{equation}\label{oddprimetorsionex1}
\big|\bigoplus_{d=0}^{p^{2}-1} T^{d/d+1}(\bar{X_{1}})_{\tors}\big|=p^{p-2}.
\end{equation}
In particular, this shows that there is no torsion in $T^{d/d+1}(\bar{X_{1}})$ for $p\leq d\leq p^{2}$.
 
\smallskip

$(ii)$ Let $A$ be a division algebra of index $8$ and exponent $2$ and $X_{2}=\SB(A)$. Again by the proof of Corollary \ref{indexsquareexponentp} with \cite[Proposition 2]{Kar95}, we get $\big|\bigoplus_{d=0}^{7} T^{d/d+1}(X_{2})_{\tors}\big|\leq 2^{2}$. Since we know by Corollary \ref{corollaryone} that $T^{2/3}(\bar{X_{2}})_{\tors}=T^{3/4}(\bar{X_{2}})_{\tors}=\Z/2\Z$ for any generic variety $\bar{X_{2}}$, we have
\begin{equation}\label{oddprimetorsionex2}
\big|\bigoplus_{d=0}^{7} T^{d/d+1}(\bar{X_{2}})_{\tors}\big|=2^{2},
\end{equation}
thus, the variety has the maximal torsion in the topological filtration.
\end{example}

A division $F$-algebra is called \emph{decomposable} if it can be written as a tensor product of two division $F$-algebras. Let $Y_{1}$ (resp. $Y_{2}$) be the corresponding Severi-Brauer variety of a decomposable algebra of index $p^{2}$ and exponent $p$ (resp. index $8$ and exponent $2$). Karpenko showed that 
\begin{equation}\label{karpenkodecomp}
\big|\bigoplus_{d=0}^{p^{2}-1} T^{d/d+1}(Y_{1})_{\tors}\big|=0\,\, \text{ and }\,\, \big|\bigoplus_{d=0}^{7} T^{d/d+1}(Y_{2})_{\tors}\big|\leq 2.
\end{equation}
in \cite[Theorem 1]{Kar952} and \cite[Theorem 2.5]{Kartri}, respectively. This result was used to prove the indecomposability of the corresponding generic algebras \cite[Corollary 5.4]{Kar}, which recovers results of Tignol \cite{Tignol87} and Amitsur-Rowen-Tignol \cite{ART}.

We remark that Example \ref{examplefordecomp} can be used to prove this result by comparing (\ref{oddprimetorsionex1}) and (\ref{oddprimetorsionex2}) with (\ref{karpenkodecomp}), thus computing the torsion in a higher quotient of the topological filtration can give a way to detect the indecomposability.

For any generic algebra $\bar{A}$ whose $p$-sequence is $(r, r-1, \cdots)$, Schofield and Van den Bergh showed the indecomposability in \cite[Theorem 2.2]{SV}, hence, the existence of indecomposable algebras of exponent bigger than a prime. Here, the assumption $\ind(\bar{A}^{\tens p})=p^{r-1}$ is crucial to use a result of Albert: $\ind(D^{\tens p})| \ind(D)/p$ for any $p$-primary algebra $D$. More precisely, as in the proof of \cite[Theorem 2.2]{SV}, if $\bar{A}\simeq B\tens C$ for algebras $B$ and $C$, then $p^{r-1}=\ind(\bar{A}^{\tens p})| \ind(B^{\tens p})\ind(C^{\tens p})| \ind(\bar{A})/p^{2}=p^{r-2}$, a contradiction. After that Karpenko showed that any generic algebra $\bar{A}$ whose $p$-seqeunce is $(r, 0)$ of length $1$ (i.e., $\bar{A}$ has a prime exponent) except the case $p=r=2$ is indecomposable in \cite[Corollary 5.4]{Kar} by comparing $T^{2/3}(\bar{X})_{\tors}$ \cite[Proposition 5.1]{Kar} with the one of a decomposable algebra, which completes the indecomposability of generic algebras.

In the following, we show that any generic algebra whose $p$-seqeunce is $(r, r-2,\cdots)$ of length $\geq r-2$ is indecomposable in Corollary \ref{coroforindecomnew}. To do this, we first compute an upper bound for the torsion subgroups in the topological filtration of decomposable algebras in Proposition \ref{newindecom}. Then, we compare it with the lower bound in Corollary \ref{corollaryone}. In contrast to earlier approaches \cite{Kar}, we use the torsion in higher degrees.

Let $A$ be a division algebra of $\ind A=p^{r}$ and $\exp(A)=p^{s}$ with $r>s$ and $r\geq 2$. Assume that $A\simeq B\tens C$   for some division algebras $B$ and $C$ of $\ind(B)=p^{k}\leq \ind(C)=p^{r-k}$. Then, the Segre embedding $\SB(B)\times \SB(M_{p^{k}}(C))\hookrightarrow \SB(M_{p^{k}}(A))$ induces 
\begin{equation}\label{Chowcomposition}
\CH^{j}(\SB(B))\tens \CH^{i}(\SB(M_{p^{k}}(C)))\to \CH^{p^{r+k}+1+i+j-p^{r}-p^{k}}(\SB(M_{p^{k}}(A)))
\end{equation}
for $0\leq j\leq p^{k}-1$ and $0\leq i\leq p^{r-k}-1$. Moreover, if $i+j+1-p^{k}\geq 0$, then by \cite[Lemma 1.12]{Mer95} we have 
\begin{equation*}
\CH^{p^{r+k}+1+i+j-p^{r}-p^{k}}(\SB(M_{p^{k}}(A)))=\CH^{i+j+1-p^{k}}(\SB(A)).
\end{equation*}
In this case, we denote by $\theta$ the composition of the map in (\ref{Chowcomposition}) and the restriction map $\CH^{i+j+1-p^{k}}(\SB(A))\to \CH^{i+j+1-p^{k}}(\SB(A)_{E})$ over a splitting field $E$.

We shall write $\overline{\CH}(\SB(A))$ for the image of the restriction map $\CH(\SB(A))\to \CH(\SB(A)_{E})$ over a splitting field $E$ and we identify $\Z\cdot h^{d}=\CH^{d}(\SB(A)_{E})$ with $\Z$, where $h$ is the class of a hyperplane of the projective space $\SB(A)_{E}$.

We will need the following lemmas, which extend \cite[Lemma 5]{Kar952} to an arbitrary exponent.

\begin{lemma}\label{decompolemone}
If $k\geq 2$, then $p^{r-1}\in \overline{\CH}^{\,pn+1}(\SB(A))$ for any $p^{r-1}-p^{k-1}\leq n\leq p^{r-1}-2$.
\end{lemma}
\begin{proof}
Let $i=p^{r}-p$, $j=pm$, and $n=p^{r-1}-p^{k-1}+m-1$ for $1\leq m\leq p^{k-1}-1$. Then, $i+j+1-p^{k}=pn+1\geq 0$. By \cite[Lemma 3]{Kar95} and (\ref{naturaltopChow}), we have
\begin{equation}\label{elementsbandc}
p^{k-1-v_{p}(m)}\in \overline{\CH}^{\, j}(\SB(B)), \,\,\,p^{r-k-1}\in \overline{\CH}^{\, i}(\SB(C)).
\end{equation}
Since the image of the tensor product of the pre-images of the elements in (\ref{elementsbandc}) under $\theta$
is 
\begin{equation}\label{lemmaoneelement}
p^{k-1-v_{p}(m)}\cdot p^{r-k-1}\cdot {{p^{k}+p-mp-2}\choose{p^{k}-1-mp}}
\end{equation}
and the $p$-adic valuation of the latter element of (\ref{lemmaoneelement}) is $1+v_{p}(m)$, the result follows.\end{proof}

\begin{lemma}\label{decompolemtwo}
For any $1\leq a\leq p^{k}-1$, $p^{r-k-v_{p}(a)}\in \overline{\CH}^{\, ap^{r-k}-p^{k}+1}(\SB(A))$. 
\end{lemma}
\begin{proof}
Let $i=ap^{r-k}$ and $j=0$ for $1\leq a\leq p^{k}-1$. Then, $1$ is contained in both $\overline{\CH}^{\, j}(\SB(B))$ and $\overline{\CH}^{\, i}(\SB(C))$. By applying the same argument as in the previous lemma, we have
\[t:={{p^{r}+p^{k}-2-ap^{r-k}}\choose{p^{k}-1}}\in \overline{\CH}^{\, ap^{r-k}-p^{k}+1}(\SB(A)).\]
Therefore, the result follows from $v_{p}(t)=r-k-v_{p}(a)$.\end{proof}

\begin{lemma}\label{decompolemthree}
Let $0\leq a\leq p^{k}-1$ be an integer such that $p^{k-1}+1\leq p^{k}b+a\leq p^{r-1}-1$ for an integer $b$ and $p^{k}b+a$ is not divisible by $p^{r-k-1}$. Then,
\[\overline{\CH}^{\, p^{k+1}b-p^{k}+pa+1}(\SB(A))\ni
\begin{cases}
p^{r-k} & \text{ if } a=0,\, p^{k-1},\\
p^{r-v_{p}(a)} & \text{ if } 1\leq a\leq p^{k-1}-1,\\
p^{r-v_{p}(a-p^{k-1})-1} & \text{ if } p^{k-1}<a\leq p^{k}-1. 
\end{cases}
\]
\end{lemma}
\begin{proof}
Let $j=0$ and $i=p^{k+1}b+pa$ for some integers $0\leq a\leq p^{k}-1$ and $b$ such that $p^{k}+p\leq i\leq p^{r}-p$ and $i$ is not divisible by $p^{r-k}$. Then, it follows from \cite[Lemma 3]{Kar95}, \cite[Lemma 1.12]{Mer95}, and (\ref{naturaltopChow}) that $1\in \overline{\CH}^{\, j}(\SB(B))$ and $p^{r-k-v_{p}(i)}\in \overline{\CH}^{\, i}(\SB(C))$. The above argument together with (\ref{Chowcomposition}) shows that
\[t:=p^{r-k-v_{p}(i)}\cdot{{p^{r}+p^{k}-2-i}\choose{p^{k}-1}}\in \overline{\CH}^{\, i-p^{k}+1}(\SB(A)).\]
Since $v_{p}((p^{k}-1)!)=1-k+(p^{k}-p)(p-1)^{-1}$, we have
\begin{align*}
v_{p}(t)&=r-k-v_{p}((p^{k}-1)!)+\sum_{u=1}^{p^{k-1}-1}v_{p}(i+pu-p^{k})\\
&=r-1-(p^{k-1}-p)(p-1)^{-1}+\sum_{u=1}^{p^{k-1}-1}v_{p}(p^{k}b-p^{k-1}+u+a),
\end{align*}
It follows from a direct computation that
\begin{equation*}
\sum_{u=1}^{p^{k-1}-1}\!\!v_{p}(p^{k}b-p^{k-1}+u+a)=\!\!
\begin{cases}
\sum v_{p}(u) & \text{ if } a=0,\, p^{k-1},\\
k-v_{p}(a)+\sum v_{p}(u) & \text{ if } 1\leq a\leq p^{k-1}-1,\\
k-v_{p}(a-p^{k-1})-1+\sum v_{p}(u) & \text{ if } p^{k-1}<a\leq p^{k}-1. 
\end{cases}
\end{equation*}
Hence, the result follows from $\sum_{u=1}^{p^{k-1}-1}v_{p}(u)=2-k+(p^{k-1}-p)(p-1)^{-1}$.\end{proof}

\begin{proposition}\label{newindecom}
Let $p$ be a prime and $A$ a division algebra of index $p^{r}$ and exponent $p^{s}$. If $A$ is decomposable, then \[\prod_{d=1}^{p^{r}-1} |T^{d/d+1}(\SB(A)_{E})/ \overline{T}^{d/d+1}(\SB(A))|\leq p^{(rp-1)p^{r-1}+s+2-r-(p^{r}-1)(p-1)^{-1}},\]
where $ \overline{T}^{d/d+1}(\SB(A))$ is the image of $\res^{d/d+1}:T^{d/d+1}(\SB(A))\to T^{d/d+1}(\SB(A)_{E})$ over a splitting field $E$ of $A$.
\end{proposition}

\begin{proof}
Let $A$ be a division algebra of $\ind(A)=p^{r}$ and $\exp(A)=p^{s}$. Assume that $A\simeq B\tens C$ for some algebras $B$ and $C$ such that $\ind(B)=p^{k}\leq \ind(C)=p^{r-k}$ for $k\geq 1$. Let $\rho(d, k)=|T^{d/d+1}(\SB(A)_{E})/ \overline{T}^{d/d+1}(\SB(A))|$ for $1\leq d\leq p^{r}-1$. As $\sum_{p\,|\, d} v_{p}(\gcd(p^{r}, d))=1-r+(p^{r}-p)(p-1)^{-1}$, we have $\sum_{p|d} v_{p}(p^{r}/\gcd(p^{r}, d))=rp^{r-1}-(p^{r}-1)(p-1)^{-1}$, thus by \cite[Lemma 3]{Kar95} we obtain
\begin{equation}\label{propchoweq1}
\prod_{p\,|\, d} \rho(d, k)\leq p^{rp^{r-1}-(p^{r}-1)(p-1)^{-1}} \text{\,\, for any } k.
\end{equation}

If $k=1$, then it follows from Lemmas \ref{decompolemtwo} and \ref{decompolemthree} that $p^{r-1}\in \overline{\CH}^{\, pn+1}(\SB(A))$ for all $1\leq n\leq p^{r-1}-2$. Therefore, we have
\begin{equation}\label{propchoweq2}
\prod_{d=pn+1} \rho(d, k)\leq p^{(r-1)(p^{r-1}-2)}.
\end{equation}
for $k=1$. By Lemmas \ref{decompolemone}, \ref{decompolemtwo} and \ref{decompolemthree}, the inequality (\ref{propchoweq2}) also holds for any $k\geq 2$. As $\CH^{1}(\SB(A))\simeq \exp(A)\cdot \Z$ and $p^{r}=\ind(A)\in \overline{\CH}^{\,d}(\SB(A))$ for any $d$, we have
\begin{equation}\label{propchoweq3}
\rho(1, k)\leq p^{s} \text{ and } \prod_{p\,\nmid\, d,\,\, d\neq pn+1}\!\!\rho(d, k)\leq p^{r(p^{r}-p^{r-1}-1-p^{r-1}+2)}\end{equation}
for $0\leq n\leq p^{r-1}-2$. Hence, the result follows from (\ref{propchoweq1}), (\ref{propchoweq2}), and (\ref{propchoweq3}).\end{proof}

\begin{corollary}$(cf.$ \cite{Kar}$)$\label{coroforindecomnew}
Let $1\leq i\leq 2$ be an integer and let $A$ be a division algebra whose $p$-sequence is $(r, r-2,\cdots)$ of length $r-i$ with $r> 1+p(p^{2}-p)^{i-2}$. Then, the corresponding generic algebra is indecomposable.
\end{corollary}

\begin{proof}
Let $p$ be a prime and $1\leq i\leq 2$ an integer. Let $r$ be an integer such that $r>1+p(p^{2}-p)^{i-2}$. Consider the division algebra $A_{i}$ whose $p$-sequence $(\alpha_{i}(k))$ is given by $\alpha_{i}(k)=r-k-1$ and $\alpha_{i}(r-i)=0$ for $1\leq k\leq r-i-1$. Let $\bar{A}_{i}$ be a generic algebra associated to $A_{i}$ and $\bar{X_{i}}=\SB(\bar{A_{i}})$. Then, we have
\begin{equation}\label{eulerphi}
v_{p}(|K((\bar{X}_{i})_{E_{i}})/K(\bar{X}_{i})|)=r\phi(p^{r})+\sum_{k=1}^{r-i-1}(r-k-1)\phi(p^{r-k}),
\end{equation}
where $E_{i}$ is a splitting field of $\bar{X_{i}}$ and $\phi$ is the Euler's phi function.

Assume that $\bar{A}_{i}$ is decomposable. Then, by Proposition \ref{newindecom} we have
\begin{equation}\label{applynewindecom}
\prod_{d=1}^{p^{r}-1} |T^{d/d+1}((\bar{X_{i}})_{E_{i}})/ \overline{T}^{d/d+1}(\bar{X}_{i})|\leq p^{(rp-1)p^{r-1}-(p^{r}-1)(p-1)^{-1}+2-i},
\end{equation}
thus, it follows from \cite[Proposition 2]{Kar95}, (\ref{eulerphi}), and (\ref{applynewindecom}) that
\begin{equation}\label{appliedtorminuss}
\big|\bigoplus_{d=0}^{p^{r}-1} T^{d/d+1}(\bar{X}_{i})_{\tors}\big|\leq p^{(i-1)(p^{2}-p-1)}.
\end{equation}
Note that the case $i=1$ of (\ref{appliedtorminuss}) was proved in \cite[Theorem 1]{Kar952}.

On the other hand, since $\lambda(1,d)=1$ for all $d=jp^{n}\geq 2$ with $1\leq j<p$ and $0\leq n<r-1$, it follows by Corollary \ref{corollaryone} that
\begin{equation}\label{corofoindeequatuib}
\big|\bigoplus_{d=0}^{p^{r}-1} T^{d/d+1}(\bar{X}_{i})_{\tors}\big|\geq p^{(p-1)(r-2)+(p-2)},
\end{equation}
which contradicts the assumption. Hence, $\bar{A}_{i}$ is indecomposable.\end{proof}


\section{Annihilators of the torsion of Severi-Brauer varieties}
In this section, we provide new upper bound for the annihilator of the torsion subgroup of $\CH^{d}(\SB(A))$ by using Proposition \ref{propsecond}. In particular, the upper bound for codimension $2$ in Corollary \ref{secondcoro} is sharp in the case where $A$ is a generic algebra having the reduced sequence of length one.

For simplicity we consider only the case of $d\leq p$. For any degree $2\leq n\leq d$ and an integer $0\leq l\leq m$, we denote by 
$$
\delta(l, n)=
\begin{cases}
\max\{(n-1)k_{m},\, \alpha(0)-\alpha(k_{m-1})-k_{m-1}\} & \text{ if } l=m,\\ 
\max\{(n-1)k_{l},\, \alpha(0)-\max\{\alpha(k_{m}), v_{p}(n)\}-k_{m}\} & \text{ if } 0\leq l\leq m-1.
\end{cases}$$
the exponent of the annihilator of $\Gamma^{n/n+1}(X)_{\tors}$ in Proposition \ref{propsecond}. 

\begin{corollary}\label{secondcoro}
Let $A$ be a $p$-primary algebra having the doubly reduced sequence $\big(\alpha(k_{i})\big)_{i=0}^{m}$. Then for any $2\leq d\leq p$, the torsion subgroup $\CH^{d}(\SB(A))_{\tors}$ is annihilated by $p$ to the
\[\sum_{2\leq n\leq t}\delta(l, n)+\sum_{t+1\leq n\leq d}\delta(0,n)\]
if $2k_{l}+\max\{\alpha(k_{l}), v_{p}(2)\}\leq \alpha(0)\leq \min\limits_{i\neq l}\{2k_{i}+\max\{\alpha(k_{i}), v_{p}(2)\}\}$ for some $0\leq l\leq m$, where $2\leq t\leq d$ is the largest integer such that $tk_{l}+\max\{\alpha(k_{l}), v_{p}(t)\}\leq \alpha(0)$.\end{corollary}

\begin{proof}
Let $2\leq n\leq d$ be an integer and $X=\SB(A)$. We use an argument in \cite[Corollary 6.8]{BNZ}. Assume that the torsion subgroup $\Gamma^{n/n+1}(X)_{\tors}$ is annihilated by an integer $N_{n}$. Then, by using the exact sequence
\[0\to \Gamma^{n/n+1}(X)\to T^{n}(X)/\Gamma^{n+1}(X)\to T^{n}(X)/\Gamma^{n}(X)\to 0\]
recursively, together with the inclusion $T^{n}(X)/\Gamma^{n}(X)\subseteq T^{n-1}(X)/\Gamma^{n}(X)$, we obtain that the torsion subgroup of $T^{d}(X)/\Gamma^{d+1}(X)$ is annihilated by $\prod_{n=2}^{d} N_{n}$, and so is $T^{d/d+1}(X)_{tors}$. As $d\leq p$, it follows from (\ref{compositecp}) that the torsion subgroup $\CH^{d}(X)_{\tors}$ is annihilated by $\prod_{n=2}^{d} N_{n}$ as well. Hence, it is enough to find such $N_{n}$.

If $2k_{i}+\max\{\alpha(k_{i}), v_{p}(2)\}\geq \alpha(0)$ for all $i$, then by Proposition \ref{propsecond} we can take $N_{n}=p^{\delta(0, n)}$ for all $2\leq n\leq d$, which proves the case where $l=0$. Now we assume that $2k_{l}+\max\{\alpha(k_{l}), v_{p}(2)\}\leq \alpha(0)$ for some nonzero $l$ and $2k_{i}+\max\{\alpha(k_{i}), v_{p}(2)\}\geq \alpha(0)$ for all $i\neq l$. Then by Proposition \ref{propsecond} we can choose 
\[N_{n}=
\begin{cases}
p^{\delta(0,n)} & \text{ if } t+1\leq n\leq d,\\
p^{\delta(l,n)} & \text{ if } 2\leq n\leq t,
\end{cases}
\]
where $t=\max\{2\leq n\leq d \,|\, nk_{l}+\max\{\alpha(k_{l}), v_{p}(n)\}\leq \alpha(0)\}$, which concludes the proof. \end{proof}

\begin{remark}\label{lastremark}
$(i)$ Assume that $A$ has the reduced sequence of length one (i.e., $m=1$). Then, it has the doubly reduced sequence. Assume that $d< p$. Then the torsion subgroup $\CH^{d}(\SB(A))_{\tors}$ is annihilated by $p$ to the
\begin{equation}\label{mequalonemd}
\begin{cases}
(d-1)(\epsilon(1)-\alpha(k_{1})) & \text{ if } \alpha(0)\leq 2k_{1}+\alpha(k_{1}),\\
(d-t)(\epsilon(1)\!-\!\alpha(k_{1}))+\frac{t(t-1)k_{1}}{2} & \text{ otherwise,}
\end{cases}
\end{equation}
where $t=\max\{2\leq n\leq d \,|\, nk_{1}+\alpha(k_{1})\leq \alpha(0)\}$.

Similarly, assume that $d=p$. Then the group $\CH^{p}(\SB(A))_{\tors}$ is annihilated by $p$ to the
\begin{equation*}
(p-1)(\epsilon(1)-\alpha(k_{1}))+\min\{0, \alpha(k_{1})-1\}
\end{equation*}
if $\alpha(0)\leq 2k_{1}+\max\{\alpha(k_{1}), v_{p}(2)\}$ and is annihilated by $p$ to the 
\begin{equation}\label{mequalonemd2}
(p-t)(\epsilon(1)-\alpha(k_{1}))+\min\{0, \alpha(k_{1})-1\}+\frac{t(t-1)k_{1}}{2}\,\, \text{ \big(resp. } \frac{p(p-1)k_{1}}{2}\big)
\end{equation}
if $2k_{1}+\max\{\alpha(k_{1}), v_{p}(2)\}< \alpha(0)< pk_{1}\!+\!\max\{\alpha(k_{1}), 1\}$ (resp. $pk_{1}+\max\{\alpha(k_{1}), 1\}\leq \alpha(0)$),
where $t=\max\{2\leq n\leq p-1 \,|\, nk_{1}+\alpha(k_{1})\leq \alpha(0)\}$.

\smallskip

$(ii)$ Consider the case codimension $d=2$. By \cite[Theorem 9.10]{GMS} and \cite[Theorem 2.1]{Pey}, the group $\CH^{2}(\SB(A))_{\tors}$ is annihilated by the order of the Rost invariant of $\gSL_{1}(A)$, which is $\exp(A)$ ($< \ind(A)=p^{\alpha(0)}$). In the case where the doubly reduced sequence of $A$ satisfies $2k_{l}+\max\{\alpha(k_{l}), v_{p}(2)\}\leq \alpha(0)\leq \min\limits_{i\neq l}\{2k_{i}+\max\{\alpha(k_{i}), v_{p}(2)\}\}$ for some $0\leq l\leq m$, Corollary \ref{secondcoro} (and Remark \ref{lastremark}(i)) gives an improved upper bound, i.e., 
\[p^{\delta(l, 2)}\leq \exp(A)\]
(see also Example \ref{exmpleforchowthm} (ii)).\end{remark}

\begin{example}\label{exmpleforchowthm}
$(i)$ Let $A$ be a central simple algebra of a prime exponent $p$ with $\ind(A)=p^{r}$ ($r\geq 2$). Let $d$ be an integer with $2\leq d\leq \min\{p, r-1\}$. Then, $A$ has the doubly reduced sequence $(\alpha(0), \alpha(1))=(r, 0)$. Hence, it follows from (\ref{mequalonemd}) and (\ref{mequalonemd2}) that for any such $d$ the torsion subgroup $\CH^{d}(\SB(A))_{\tors}$ is annihilated by $p^{d(d-1)/2}$.

Now consider a generic algebra $D$ of index $p^{r}$ and exponent $p^{s}$ with $1\leq s< r$ whose $p$-sequence is $(r, r-1,\ldots, r-s+1, 0)$. Then, it has the doubly reduced sequence $(\alpha(0), \alpha(s))=(r, 0)$ for $d\leq p$. By the same argument, we see that for any $2\leq d\leq \min\{p, \lceil r/s\rceil-1\}$ the torsion subgroup $\CH^{d}(\SB(D))_{\tors}$ is  annihilated by $(p^{d(d-1)/2})^{s}$.

\smallskip

$(ii)$ Let $A$ be a central simple algebra of $\exp(A)=p^{27}$ whose reduced sequence is $(\alpha(0), \alpha(2), \alpha(4), \alpha(5))=(30, 27, 24, 22)$. Then, by Lemma \ref{firstlemma} this sequence is doubly reduced for any $d\leq p$. As $\alpha(k_{i})+2k_{i}\geq \alpha(0)$ for all $1\leq i\leq m=3$, it follows by Corollary \ref{secondcoro} that for any $2\leq d\leq p$ the torsion subgroup $\CH^{d}(\SB(A))_{\tors}$ is annihilated by $p^{3d-3}$.

\end{example}

\paragraph{\bf Acknowledgments.} 
This work was partially supported by an internal fund from KAIST, TJ Park Junior Faculty Fellowship of POSCO TJ Park Foundation, and National Research Foundation of Korea (NRF) funded by the Ministry of Science, ICT and Future Planning (2013R1A1A1010171).

\end{document}